\documentclass[12pt]{article}

\usepackage{graphicx}
\usepackage[a4paper, hmargin={2.5cm,2.5cm},vmargin={3.2cm,3.2cm}]{geometry}
\usepackage[font=small,labelfont=bf,width=0.9\textwidth]{caption} 

\usepackage[T1]{fontenc}
\usepackage[utf8]{inputenc}
\usepackage[english]{babel}

\usepackage{amsmath,amssymb,amsfonts}
\usepackage{mathrsfs}
\usepackage{amsthm,latexsym}
\usepackage{etoolbox}
\usepackage{extarrows}

\usepackage{cite}
\usepackage{paralist}
\usepackage{tikz}
\usetikzlibrary{shapes}
\usepackage{pgfplots}
\pgfplotsset{compat=newest}
\usepgflibrary{plotmarks}

\usepackage{filecontents}

\usepackage{twoopt}

\usepackage{url}
\usepackage[bf,compact,pagestyles,small]{titlesec}
\titlespacing*{\section}{0pt}{14pt}{4pt}
\titlespacing*{\subsection}{0pt}{8pt}{3pt}

\usepackage{fancyhdr,lastpage}
\def\maketimestamp{\count255=\time
\divide\count255 by 60\relax
\edef\thetime{\the\count255:}%
\multiply\count255 by-60\relax
\advance\count255 by\time
\edef\thetime{\thetime\ifnum\count255<10 0\fi\the\count255}
\edef\thedate{\number\day-\ifcase\month\or Jan\or Feb\or Mar\or
             Apr\or May\or Jun\or Jul\or Aug\or Sep\or Oct\or
             Nov\or Dec\fi-\number\year}
\def\timstamp{\hbox to\hsize{\tt\hfil\thedate\hfil\thetime\hfil}}}
\maketimestamp
\fancypagestyle{plain}{%
\fancyhf{} 
\fancyfoot[C]{\small \thepage{} of \pageref{LastPage}}}

\numberwithin{equation}{section}  
\allowdisplaybreaks[4]
\newtheorem{theorem}{Theorem} 
\newtheorem{lemma}[theorem]{Lemma}

\theoremstyle{definition}
\newtheorem*{conjecture*}{Conjecture} 
\newtheorem{example}{Example}
\theoremstyle{remark}

\DeclareMathOperator{\Span}{span} %
\DeclareMathOperator*{\esssup}{ess\,sup} %
\DeclareMathOperator{\ft}{\mathcal{F}}
\DeclareMathOperator{\exponential}{e}

\newcommandtwoopt{\gaborG}[3][a][b]{\mathcal{G}(#3,#1,#2)} 

\newcommand{\myexp}[1]{\exponential^{#1}}

\newcommand*{\numbersys}[1]{\ensuremath{\mathbb{#1}}}
\newcommand*{\C}{\numbersys{C}}
\newcommand*{\R}{\numbersys{R}}

\newcommand*{\Z}{\numbersys{Z}}

\newcommand*{\N}{\numbersys{N}}

\newcommand*{\cF}{\mathcal{F}}

\newcommand*{\frameset}{\mathscr{F}}

\newcommand{\itvcc}[2]{\ensuremath{\left[{#1},{#2}\right]}} %
\newcommand{\itvco}[2]{\ensuremath{\left[{#1},{#2}\right)}} %
\newcommand{\itvcos}[2]{\ensuremath{\lbrack{#1},{#2})}} %
\newcommand{\abs}[1]{\ensuremath{\left\lvert#1\right\rvert}}

\newcommand{\absbig}[1]{\ensuremath{\bigl\lvert#1\bigr\rvert}}

\newcommand{\norm}[2][]{\ensuremath{\left\lVert#2\right\rVert_{#1}}}

\newcommand{\innerprod}[3][]{\ensuremath{\left\langle #2,#3\right\rangle_{\! #1}}}

\newcommand{\set}[1]{\ensuremath{\left\lbrace{#1}\right\rbrace}}

\newcommand{\setprop}[2]{\ensuremath{\left\lbrace{#1} : {#2}\right\rbrace}}

\newcommand{\setpropsmall}[2]{\ensuremath{\lbrace{#1} : {#2}\rbrace}}

\newcommand*{\defeq}{\stackrel{\mathrm{def}}{=}}

\usepackage{hyperref}
\hypersetup{
 pdfview={FitH},
 pdfstartview={FitH},
 pdfauthor = {Jakob Lemvig} 
 pdftitle = {},
 pdfkeywords = {Hermite, frame, frame set, Gabor system},
 pdfcreator = {LaTeX with hyperref package},
 pdfproducer = PDFlatex}

\makeatletter
\def\blfootnote{\xdef\@thefnmark{}\@footnotetext}
\def\subjclass{\xdef\@thefnmark{}\@footnotetext}
\long\def\symbolfootnote[#1]#2{\begingroup%
\def\thefootnote{\fnsymbol{footnote}}\footnote[#1]{#2}\endgroup}
\if@titlepage
  \renewenvironment{abstract}{%
      \titlepage
      \null\vfil
      \@beginparpenalty\@lowpenalty
      \begin{center}%
        \bfseries \abstractname \eqref{eq:F-of-Zak} and the
fact that $\hat{h}_n= -h_n$
        \@endparpenalty\@M
      \end{center}}%
     {\par\vfil\null\endtitlepage}
\else
  \renewenvironment{abstract}{%
      \if@twocolumn
        \section*{\abstractname}%
      \else
        \small
        \list{}{%
          \settowidth{\labelwidth}{\textbf{\abstractname:}}
          \setlength{\leftmargin}{50pt}
          \setlength{\rightmargin}{50pt}
          \setlength{\itemindent}{\labelwidth}
          \addtolength{\itemindent}{\labelsep}
        }
        \item[\textbf{\abstractname:}]

      \fi}
      {\if@twocolumn\else\endlist\fi}
\fi
\makeatother

\begin{document}

\title{On some Hermite series identities and their applications to
  Gabor analysis}

\date{\today}

 \author{Jakob Lemvig\footnote{Technical University of Denmark, Department of Applied Mathematics and Computer Science, Matematiktorvet 303B, 2800 Kgs.\ Lyngby, Denmark, E-mail: \protect\url{jakle@dtu.dk}}\phantom{$\ast$}} 


\maketitle

\thispagestyle{plain}
 \begin{abstract} 
   We prove some infinite series identities for the Hermite functions. From these identities we disprove the Gabor frame set conjecture for Hermite functions of order $4m+2$ and $4m+3$ for $m\in \{0\} \cup \N$. The results hold not only for Hermite functions, but for two large classes of eigenfunctions of the Fourier transform associated with the eigenvalues $-1$ and $i$, and the results indicate that the Gabor frame set of all such functions must have a rather complicated structure.
 \end{abstract}

\section{Introduction} 
\label{sec:non-frame-property}


Since John von Neumann's claim of completeness of the
 coherent state subsystems generated by the 
 Gaussian in his work on 
quantum mechanics \cite{MR0223138}, it has been of interest in mathematical physics and
analysis to determine when the set of coherent
states $\gaborG{g}:= \set{\myexp{2\pi i bm \cdot}g(\cdot-ak)}_{k,m\in
  \Z}$ is complete in various function spaces, e.g., $L^2(\R)$.  In
engineering, $\gaborG{g}$ is the so-called Gabor system
generated by the window function $g \in L^2(\R)$ with time-frequency
shifts along the lattice $a\Z \times b\Z$ in phase space. 
 For most
applications in signal processing and functional analysis,
completeness of $\gaborG{g}$ is nowadays not considered to be
sufficient; for instance, to guarantee unconditionally $L^2$-convergent and
stable expansions of functions in $L^2(\R)$ and to provide
characterizations of classical function spaces, one needs a stronger
property of $\gaborG{g}$, namely that the Gabor system constitutes a
frame for $L^2(\R)$, i.e, existence of constants $A,B>0$, termed frame
bounds, such that 
\begin{equation}
A \norm{f}^2 \le \sum_{k,m \in \Z} \abs{\innerprod{f}{\myexp{2\pi i bm
      \cdot}g(\cdot-ak)}}^2 \le B \norm{f}^2 \quad \text{for all } f
\in L^2(\R).\label{eq:frame-def}
\end{equation}

In this work we are interested in the frame properties of Gabor
systems generated by Hermite functions.
 We define the $n$th Hermite function  $h_n$ by
\[ 
h_n(x) = (c_n)^{-1/2} \myexp{\pi x^2} \left(\frac{d^n}{dx^n} \myexp{-2\pi x^2}\right) ,
\]
where $c_n = (2\pi)^n 2^{n-1/2} n!$ 
for $n \in \N \cup \{0\}$. The class of Hermite functions forms a natural
continuation of the study of von Neumann~\cite{MR0223138} and
Gabor~\cite{GaborTheory1946} as it contains the Gaussian
as a special case, $n=0$.
The \emph{frame set} of a window function $g \in L^2(\R)$, denoted by
$\frameset{(g)}$, is the parameter
values $(a,b)\in\mathbb{R}_+^2$ for which the associated Gabor system
$\gaborG{g}$ is a frame for $L^2(\R)$. Hence, we will study the set $\frameset{(h_n)}$, or to be
more precise, properties of its compliment. That is, following \cite{MR3232589}, we will ask what
prevents $\gaborG{g}$ from generating a frame? Our answers will show that the
Gabor frame set of Hermite functions must have a rather complicated
structure. Indeed, we will derive new obstructions
of the frame property for two classes of eigenfunctions of the Fourier
transform associated with the eigenvalue $-1$ and $i$, respectively,
which, in particular, disproves a conjecture on Hermite functions by Gr\"ochenig~\cite{MR3232589}. 

To understand Gr\"ochenig's conjecture, let us recall what is known
about $\frameset(h_n)$. Since Hermite functions have exponential decay
in time and frequency domain, it is known, see e.g., \cite{MR3232589},
that the upper frame bound holds, that the set $\frameset(h_n)$ is
open in $\R^2$ and that $\frameset(h_n) \subset \setprop{(a,b)\in \mathbb{R}^2_+
}{ ab < 1}$. For the Gaussian $h_0$, the necessary condition $ab<1$ for the
frame property is also
sufficient. This important result was
conjectured by Daubechies and Grossmann~\cite{MR924682} and proved by 
Lyubarskii~\cite{MR1188007} and by Seip and Wallst\'en \cite{MR1173117,MR1173118}. 
The proof relies on analytic properties of the short-time Fourier
transform of the Gaussian and the fact that
the Bargmann transform of an $L^2$-function is analytic. 
In \cite{MR2529475,MR2292280} Gr\"ochenig and Lyubarskii obtained the
following generalization: for any pair $(a,b)$ in $\mathbb{R}^2_+$ with $ab < \frac{1}{n+1}$,
the Gabor family $\gaborG{h_n}$ is a frame. Finally, Lyubarskii and
Nes~\cite{MR3027914} proved that the frame set of any sufficiently nice, \emph{odd} window
function, in particular, $h_{2m+1}$, $m \in \N \cup \{0\}$, cannot contain the hyperbolas
$ab= \tfrac{p}{p+1}$ for any $p \in \N$.
As no other obstructions for the frame property of $h_n$ was known,
this led Gr\"ochenig~\cite{MR3232589} to conjecture that 
 the frame set for the even Hermite
functions is the largest possible set $\frameset{(h_{2m})}=\setprop{(a,b)\in
  \R^2_+}{ab<1}$, and that the frame set for the odd Hermite functions
is $\frameset{(h_{2m+1})}=\setpropsmall{(a,b)\in
  \R^2_+}{ab<1, ab \neq \tfrac{p}{p+1}, p \in \N}$, $m \in \N \cup \{0\}$.
The conjecture is true for $h_0$ by the above mentioned results. The
conjecture for $h_1$ is due to
Lyubarskii and Nes~\cite{MR3027914}, and this paper will not shed new
light on this case. However, our results show that the conjecture is false for
$h_n$ with $n=4m+2$ and $n=4m+3$, $m \in \N\cup\{0\}$. We also give numerical
evidence in Section~\ref{sec:numer-exper} that it is false for $n=4$ and $n=5$ which leads us to
believe that the conjecture is also false for $n=4m$ and $n=4m+1$,
whenever $m>0$.  

Our proofs are based on Zak transform methods and certain infinite
series identities which are of independent interest. As an example, we
will show that $h_{4m+2}$, $m \in \N\cup\{0\}$, satisfies
\begin{equation}
\sum_{k\in \Z} (-1)^k h_{4m+2}(\sqrt{2}(k+\tfrac{p}{4}))= 0 \quad \text{for $p \in \set{1,3}$}.
\label{eq:h2-identity}
\end{equation}
For $m=0$ the identity concerns $h_2$, and it reads, for $p=1$,
\begin{equation}
 \sum_{k\in \Z} (-1)^k (8\pi (k+\tfrac{1}{4})^2-1)
\myexp{-2\pi\bigl(k+\tfrac{1}{4}\bigr)^2}= 0 ,
\label{eq:h2-identity-expl}
\end{equation}
which is illustrated in Figure~\ref{fig:hermite}. As we shall see in
Section~\ref{sec:some-infinite-series}, the identities in
\eqref{eq:h2-identity} are even true for any sufficiently nice
function that is an eigenfunction of the Fourier transform with
eigenvalue $-1$.
\begin{figure}
  \centering
\includegraphics{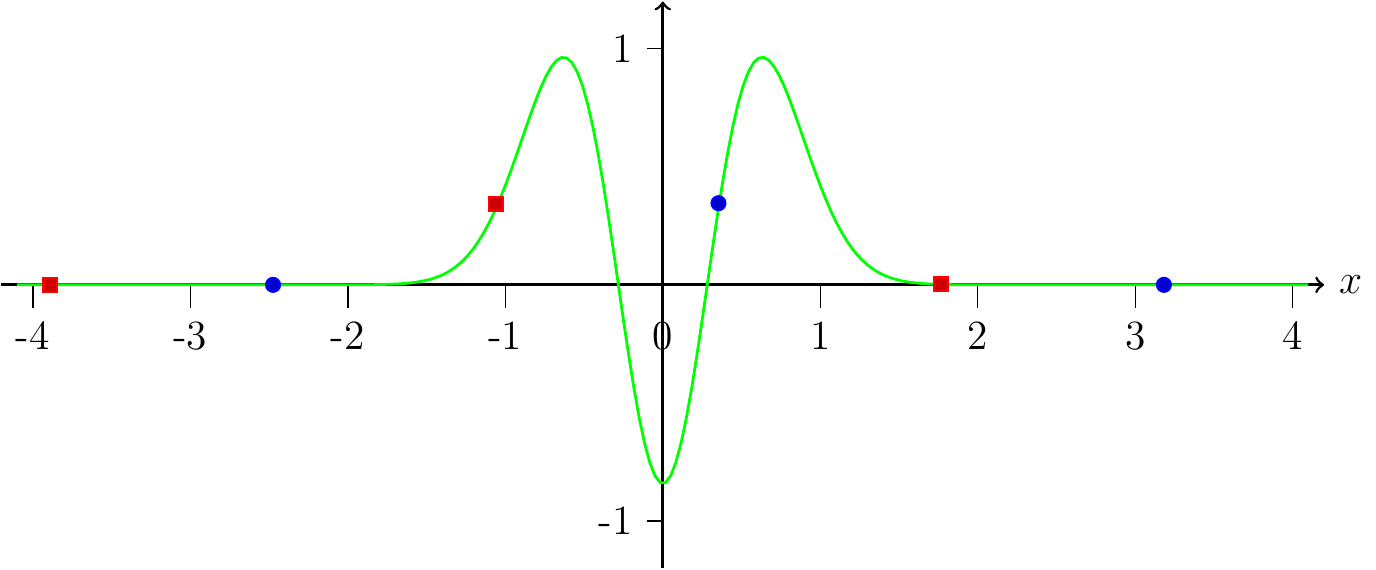}
\caption{The graph of $h_2$ and an illustration of the identity
\eqref{eq:h2-identity-expl}, where the samples for even and odd $k \in \Z$ are marked with
  blue circles and red squares, respectively. Note that the sampling has no
  simple symmetries, e.g., $h_2(-\sqrt{2}\,3/4)\neq h_2(\sqrt{2}/4)$.}
\label{fig:hermite}
\end{figure}
From the identity \eqref{eq:h2-identity} it follows that the Zak
transform $Z_{\sqrt{2}}$ of $h_{4m+2}$ has two zeros in $\itvco{0}{1}^2$, located one-half
apart on a horizontal line. By standard Zak transform methods in Gabor analysis,
detailed in Section~\ref{sec:obstruc},
it follows that
$\gaborG[1/\sqrt{2}][1/\sqrt{2}]{h_{4m+2}}$ is not a
frame. Note that it is not our focus to give a detailed analysis of
the frame set of specific Hermite functions, e.g.,
$\frameset{(h_2)}$. Instead, we are interested in determining values of $a$ and
$b$ for which $\gaborG{g}$ fails to be a frame for every nice window
$g$ in, e.g., the class of eigenfunctions of the Fourier transform
associated with the eigenvalue $-1$ to which all Hermite
functions of the
form $h_{4m+2}$, $m\in \N \cup\{0\}$, belong. Previously, not a single
obstruction for the frame property was known for any of the functions in this class.


\section{Preliminaries}
\label{sec:preliminaries}

We begin by recalling some properties of
the Hermite functions and the Zak transform.

\subsection{Hermite functions}
\label{sec:hermite-functions}

Hermite functions arise in many different
contexts, e.g., as eigenfunctions of the Hermite operator
$H=-\frac{d^2}{dx^2}+(2\pi x)^2$.
What is more important for us is that the Hermite functions are also
eigenfunctions for the Fourier transform:
\[ 
 \hat{h}_n(\gamma) = (-i)^n h_n(\gamma)  \quad a.e.\ \gamma \in \R.
\] 
Here, the Fourier transform is defined for $f \in L^1(\R)$ by 
\[
\ft
f(\xi)=\hat f(\gamma) = \int_{\R} f(x)\myexp{-2 \pi i
  \gamma x} \mathrm{d}x
  \]
with the usual extension to $L^2(\R)$. We let $H_j$, $j=0,1,2,3$,
denote the eigenspace of the Fourier transform corresponding to the
eigenvalue $(-i)^j$. More specifically, since $\set{h_n}_{n=0}^\infty$
is an orthonormal basis for $L^2(\R)$,
\[ 
H_j = \ker (\ft-(-i)^jI) = \overline{\Span{\setprop{h_{4m+j}}{m \in \N}}} =
\setprop{\sum_{m \in \N} c_m h_{4m+j}}{(c_m)\in \ell^2(\N)}.
\]
By $\cF^2\{f(x)\}=f(-x)$, it follows that any function in $H_j$, $j=0,2$, is
even and that any function in $H_j$, $j=1,3$, is odd.

Since the Fourier transform is a unitary operator, it preserves the
frame property, that is, the system $\gaborG{g}$ is a frame if and
only if the Fourier transform of the system $\gaborG[b][a]{\hat{g}}$
is a frame. Since the eigenvalue of the Hermite functions is of
modulus one, we immediately have the following simple result. It
implies that the frame set of Hermite functions is symmetric about the line $a=b$,
i.e., $(a,b)\in \frameset(h_n)$ if and only if $(b,a)\in \frameset(h_n) $.
\begin{lemma}
  Let $a,b>0$, $A,B>0$, and let $g \in H_j$ for some $j=0,1,2,3$. Then the
  following are equivalent:
  \begin{enumerate}[(i)]
  \item $\gaborG{g}$ is a frame with bounds $A$ and $B$,
  \item $\gaborG[b][a]{g}$ is a frame with bounds $A$ and $B$.
  \end{enumerate}
\end{lemma}

\subsection{The Zak transform}
\label{sec:zak-transform}

For any $\lambda>0$, the Zak transform of a function $f \in L^2(\R)$ is defined as
\begin{equation}\label{eq:zakTransform}
\left(Z_{\lambda}f\right)(x,\gamma)
= \sqrt{\lambda}\sum_{k\in\mathbb{Z}} f(\lambda(x+
  k))\myexp{-2\pi i k \gamma}, \quad a.e.\ x, \gamma \in \mathbb{R},
\end{equation}
with convergence in  $L^2_\mathrm{loc}(\R)$. The Zak transform
$Z_\lambda$ is a unitary map of $L^2(\R)$ onto $L^2(\itvco{0}{1}^2)$, and
 it has the following quasi-periodicity:
\[
Z_\lambda f(x+1,\gamma)= \myexp{2\pi i\gamma} Z_\lambda f(x,
\gamma),  \quad  Z_\lambda f(x, \gamma +1) = Z_\lambda f(x, \gamma) \quad \text{for 
a.e. } x,\gamma \in \R.
\]  
The Zak transform has been used by Weil~\cite{MR0165033} in harmonic
analysis on locally compact abelian groups, by
Gel'fand~\cite{MR0073136} in the study of Schr\"odinger's equation,
and by Zak\cite{MR1478343} in solid state physics. For a systematic
treatment of the Zak transform and its use in applied mathematics, we
refer to the paper by Janssen~\cite{MR947891}.  Recent
applications in Gabor analysis include
\cite{GrochenigCompleteness2015,MR3218799,MR3393698}.

The Zak transform inherits symmetries of the function $f$. The
following basic lemma will be used several times in the later
sections. The Wiener space $W(\R)$ consists of functions $g \in L^\infty(\R)$
for which $\sum_{k \in \Z} \esssup_{x \in
  \itvcc{0}{1}}\abs{g(x+k)}<\infty$. The assumption that $f$ belongs
to $W(\R)$ and is continuous in
Lemma~\ref{lem:Zak-symm} implies that $Z_\lambda f$ is continuous
which guarantees that the identities in the lemma hold pointwise.

\begin{lemma}
\label{lem:Zak-symm}
Let $m \in \Z$ and $\lambda>0$. Assume that $f \in W(\R)$ is continuous. 
\begin{enumerate}[(i)]
\item 
  If $f$ is an even function, then
  \[ Z_\lambda f(x,\gamma) = Z_\lambda f(-x,-\gamma) \qquad \text{for
    all } x,\gamma \in \R .\] In particular, $Z_\lambda f(x,\tfrac{m}{2}) = (-1)^m Z_\lambda
  f(1-x,\tfrac{m}{2})$ and
 \[ Z_\lambda f(x,\gamma) = 0 \qquad (x,\gamma) \in \Z^2 + (\tfrac12,\tfrac12).\] 
\item If $f$ is an odd function, then
  \[ Z_\lambda f(x,\gamma) = - Z_\lambda
  f(-x,-\gamma) \qquad \text{for
    all } x,\gamma \in \R . \] In particular, $Z_\lambda f(x,\tfrac{m}{2}) = (-1)^{m+1} Z_\lambda
  f(1-x,\tfrac{m}{2})$ and
 \[ Z_\lambda f(x,\gamma) = 0 \qquad (x,\gamma) \in \tfrac12 \Z^2
 \setminus \left(\Z^2 + (\tfrac12,\tfrac12)\right).\] 
\end{enumerate}
\end{lemma}

By quasi-periodicity, the function $Z_\lambda f$ on $\R^2$ is
determined by its values on $\itvco{0}{1}^2$. Hence, if
$Z_\lambda f(x_0,\gamma_0)=0$ for some $(x_0,\gamma_0) \in \R^2$, then
$Z_\lambda f(x,\gamma)=0$ for all $(x,\gamma)\in \Z^2 +
(x_0,\gamma_0)$. For this reason we will often only explicitly mention the
zeros of $Z_\lambda f$ on $\itvco{0}{1}^2$.


If $f \in W(\R)$ and $\hat{f} \in W(\R)$, it follows by an application
of Poisson summation formula, see e.g., \cite{MR947891} or \cite[Proposition 8.2.2]{MR1843717}, that 
\begin{equation}
  \label{eq:F-of-Zak}
  Z_\lambda f(x,\gamma) = \myexp{2\pi i x \gamma}
  Z_{1/\lambda}\hat{f}(\gamma,-x) \quad \text{for all } x,\gamma \in \R,
\end{equation}
with absolute convergence of the series . 
In particular, this relation holds for any function $f$ in $H_j \cap
W(\R)$ for $j=0,1,2,3$.  Note that any
function $f$ in $H_j \cap
W(\R)$ is continuous since $\hat{f} \in W(\R) \subset L^1(\R)$.


\section{Some infinite series identities}
\label{sec:some-infinite-series}

The infinite series identities for Hermite functions derived in this
section will play a crucial role in the counterexamples in Section~\ref{sec:obstruc}. The
identities are of independent interest and can be formulated as
multiple zeros of the Zak transform. We remark that it is not
difficult to find a single zero of the Zak transform of Hermite functions, see, e.g.,
Lemma~\ref{lem:Zak-symm}. We will find $k$ zeros of $Z_\lambda h_n(x,\gamma)$
for a fixed value of $\gamma$, each $1/(k+1)$ apart with respect to
the $x$ variable, which is a much harder task that depends delicately
on the parameter $\lambda$. 

\begin{lemma}
\label{lem:h2-identities}
Let $n=4m+2$ for some $m \in \N \cup \{0\}$. Then
\begin{equation}
Z_{\sqrt{2}}h_n(\tfrac{p}{4},\tfrac12) \defeq  2^{1/4} \sum_{k\in \Z} (-1)^k h_{n}(\sqrt{2}(k+\tfrac{p}{4}))=0 \quad \text{for $p \in \set{1,3}$},
\label{eq:h-even-sqrt-2}
\end{equation}
and
\begin{equation}
  Z_{\sqrt{3}}h_n(\tfrac{p}{6},\tfrac12) \defeq 3^{1/4} \sum_{k\in \Z}
  (-1)^k h_{n}(\sqrt{3}(k+\tfrac{p}{6}))=0 \quad \text{for $p \in
    \set{1,5}$}. \label{eq:h-even-sqrt-3}
\end{equation}
\end{lemma}
\begin{proof}
We first prove the assertions in \eqref{eq:h-even-sqrt-2}. 
Let
$p=1$. Since the sum in \eqref{eq:h-even-sqrt-2} converges absolutely,
we can split the sum in even and odd
indices $k \in \Z$. Hence, proving \eqref{eq:h-even-sqrt-2} is
equivalent to proving: 
\[
 \sum_{k\in \Z}  h_{n}(\sqrt{2}(2k+\tfrac{1}{4}))=  \sum_{k\in \Z}  h_{n}(\sqrt{2}(2k+\tfrac{5}{4})).
\]   
In terms of the Zak transform, we need to prove that
\begin{equation}
Z_{2^{3/2}}h_n (\tfrac18,0) = Z_{2^{3/2}}h_n (\tfrac58,0). \label{eq:Zak-sqrt-8}
\end{equation}
We first consider the left hand side. By \eqref{eq:F-of-Zak} and the
fact that $\hat{h}_n= -h_n$, we obtain
\[ 
Z_{2^{3/2}}h_n (\tfrac18,0) = Z_{2^{-3/2}}\hat{h}_n (0,-\tfrac18) \defeq -2^{-3/4} \sum_{k \in \Z} h_n(2^{-3/2}k)
\myexp{2\pi i k/8}.
\]
Substituting $k \in \Z$ for $8m+\ell$, where $m\in \Z$ and $\ell
=0,1,\dots, 7$, we find that
\begin{align}
Z_{2^{3/2}}h_n (\tfrac18,0) 
 &= -2^{-3/4} 
\sum_{\ell=0}^7 \sum_{m \in \Z} h_n(2^{3/2}(m+\frac{\ell}{8}))
\myexp{2\pi i \ell/8} \nonumber \\ 
&= -2^{-3/2}\sum_{\ell=0}^7 Z_{2^{3/2}}h_n (\tfrac{\ell}{8},0) 
\myexp{2\pi i \ell/8} \label{eq:2}
\end{align}
The odd terms over $\ell$ sum to:
\begin{align*}
  \sum_{\ell\in\{1,3,5,7\}} Z_{2^{3/2}}h_n
  (\tfrac{\ell}{8},0) \myexp{2\pi i \ell/8} =& Z_{2^{3/2}}h_n
  (\tfrac{1}{8},0) (\myexp{2\pi i /8}+\myexp{2\pi i 7/8}) +
  Z_{2^{3/2}}h_n (\tfrac{5}{8},0) (\myexp{2\pi i 3/8}+\myexp{2\pi i
    5/8})\\ =& \sqrt{2} Z_{2^{3/2}}h_n (\tfrac{1}{8},0) - \sqrt{2}
  Z_{2^{3/2}}h_n (\tfrac{5}{8},0),
\end{align*}
where we have used Lemma~\ref{lem:Zak-symm}. Similarly, we find that
\begin{equation}
Z_{2^{3/2}}h_n (\tfrac58,0) 
 = -2^{-3/2} 
\sum_{\ell=0}^7 Z_{2^{3/2}}h_n
  (\tfrac{\ell}{8},0) 
\myexp{2\pi i 5\ell/8} \label{eq:3}
\end{equation}
where the odd terms over $\ell$ sum to:
\begin{align*}
  \sum_{\ell\in\{1,3,5,7\}} Z_{2^{3/2}}h_n
  (\tfrac{\ell}{8},0) \myexp{2\pi i 5\ell /8} =& Z_{2^{3/2}}h_n
  (\tfrac{5}{8},0) (\myexp{2\pi i /8}+\myexp{2\pi i 7/8}) +
  Z_{2^{3/2}}h_n (\tfrac{1}{8},0) (\myexp{2\pi i 3/8}+\myexp{2\pi i
    5/8})\\ =& \sqrt{2} Z_{2^{3/2}}h_n (\tfrac{5}{8},0) - \sqrt{2}
  Z_{2^{3/2}}h_n (\tfrac{1}{8},0).
\end{align*}
Note that $\ell \equiv 5\ell \pmod 8$ for even $\ell \in 2\Z$. 
Thus, if we subtract the two right hand sides of \eqref{eq:2} and
\eqref{eq:3},
the
even terms over $\ell=0,2,4,6$ cancel out. Hence,
\[
Z_{2^{3/2}}h_n (\tfrac18,0) - Z_{2^{3/2}}h_n (\tfrac58,0) = -(Z_{2^{3/2}}h_n (\tfrac18,0) - Z_{2^{3/2}}h_n (\tfrac58,0)) .
\]
However, this is only possible if \eqref{eq:Zak-sqrt-8} holds which was what we had to prove. This
completes the proof of the case $p=1$.  

For the case $p=3$, note that, by Lemma~\ref{lem:Zak-symm}, 
\[ 
Z_{\sqrt{2}}h_n(\tfrac{1}{4},\tfrac12)  =-Z_{\sqrt{2}}h_n(\tfrac{3}{4},\tfrac12),
\]
hence the identity follows from the case $p=1$. 


The proof of  \eqref{eq:h-even-sqrt-3} goes along the same lines as
the proof of \eqref{eq:h-even-sqrt-2}; the details are left for the reader.
\end{proof}

\begin{lemma}
\label{lem:h3-identities}
Let $n=4m+3$ for some $m \in \N \cup \{0\}$ and let $s \in \{2,3,4\}$. Then
\[
Z_{\sqrt{s}}h_n(\tfrac{p}{s},0) \defeq s^{1/4} \sum_{k\in \Z} h_{n}(\sqrt{s} (k+\tfrac{p}{s}))=0
 \quad \text{for $p \in \set{0,1,\dots, s-1}$.}
\]
\end{lemma}
\begin{proof}
We will only prove the case $s=3$ as the other cases are similar. 
  For $p=0$ the identity follows from the fact that $h_n$ is an odd
  function. For $p=1$ we have, using  \eqref{eq:F-of-Zak} and
  $\hat{h}_n= i h_n$,
  \begin{multline}
    3^{-1/4} Z_{\sqrt{3}}h_n(\tfrac{1}{3},0) =
    3^{-1/4} Z_{\tfrac{1}{\sqrt{3}}}\hat{h}_n(0,-\tfrac{1}{3}) = 3^{-1/2} \sum_{k\in \Z}
    i h_{n}(\tfrac{1}{\sqrt{3}}k) \myexp{2\pi i k/3} \\ 
= 3^{-1/2} i  \left( \sum_{m\in \Z}
    h_{n}(\sqrt{3}m) + \sum_{m\in \Z}
    h_{n}(\sqrt{3}(m+\tfrac{1}{3})) \myexp{2\pi i /3} + \sum_{m\in \Z}
    h_{n}(\sqrt{3}(m-\tfrac{1}{3})) \myexp{-2\pi i /3} \right),  
\label{eq:5}
\end{multline}
where we have substituted $k$ for $3m+\ell$ with
  $m \in \Z$ and $\ell \in \set{-1,0,1}$,
Since $h_n$ is odd, it follows directly that $\sum_{m\in \Z}
    h_{n}(\sqrt{3}m)=0$. By yet another symmetry argument (e.g.,
    Lemma~\ref{lem:Zak-symm}), we also see that
\[ \sum_{m\in \Z}
    h_{n}(\sqrt{3}(m-\tfrac{1}{3})) = \sum_{m\in \Z}
    h_{n}(\sqrt{3}(m+\tfrac{2}{3})) = -\sum_{m\in \Z}
    h_{n}(\sqrt{3}(m+\tfrac{1}{3})).
\]
Continuing the computation in \eqref{eq:5} yields
  \begin{align*}
  \sum_{k\in \Z} h_{n}(\sqrt{3} (k+\tfrac{1}{3})) &\defeq   3^{-1/4}
  Z_{\sqrt{3}}h_n(\tfrac{1}{3},0) = 3^{-1/2} i (\myexp{2\pi i /3}-\myexp{-2\pi i /3}) \sum_{m\in \Z}
    h_{n}(\sqrt{3}(m+\tfrac{1}{3}))\\ &= - \sum_{m\in \Z}
    h_{n}(\sqrt{3}(m+\tfrac{1}{3})), 
  \end{align*}
where we use that $\myexp{2\pi i /3}-\myexp{-2\pi i /3}=i \sqrt{3}$. Thus 
$\sum_{m\in \Z}
    h_{n}(\sqrt{3}(m+\tfrac{1}{3})) =0$ which completes the case
    $p=1$. 

Consider now $p=2$. By Lemma~\ref{lem:Zak-symm} we have 
\[ 
Z_{\sqrt{3}}h_n(\tfrac{1}{3},0)  =-Z_{\sqrt{3}}h_n(\tfrac{2}{3},0),
\]
hence the assertion for $p=2$ follows from the case $p=1$. 
\end{proof}

  Note that the only property of $h_n$ used in the proof of the above two
  lemmas is that $h_n$ is an eigenfunction of the Fourier transform
  associated with the eigenvalue $-1$ and $i$, respectively, for which
  Poisson summation formula~\eqref{eq:F-of-Zak} holds pointwise with absolute convergence. Recall that
   functions in $H_2 \cap W(\R)$ are even and continuous, while
   functions in $H_3 \cap W(\R)$ are odd and continuous. Therefore, we
 can formulate the following extension of the results in this section using Lemma~\ref{lem:Zak-symm}.
   \begin{lemma}
\label{lem:extension-identities}
\begin{enumerate}[(i)]
\item For $g \in H_2 \cap W(\R)$, we have:
\[ 
  Z_{\sqrt{2}} g (x,\gamma) = 0 \quad \text{for } (x,\gamma) \in
  (\tfrac14 \Z \setminus \Z) \times (\Z +\tfrac12),
\]
and
\[ 
  Z_{\sqrt{3}} g (x,\gamma) = 0 \quad \text{for } (x,\gamma) \in
  (\tfrac13 \Z + \tfrac16) \times (\Z +\tfrac12).
\]
\item For $g \in H_3 \cap W(\R)$ and $s\in \{2,3,4\}$, we have:
\[ 
  Z_{\sqrt{s}} g (x,\gamma) = 0 \quad \text{for } (x,\gamma) \in
  \tfrac1s \Z \times \Z .
\]
\end{enumerate}
   \end{lemma}

\section{New obstructions of the frame property}
\label{sec:obstruc}

For rationally oversampled Gabor systems, i.e., $\mathcal{G}(g,a,b)$ with
\[
ab \in \mathbb{Q}, \quad ab=\frac{p}{q} \quad \gcd(p,q)=1,
\]
we define column vectors $\phi^g_\ell(x,\gamma) \in \C^p$ for $\ell
\in \set{0,1, \dots, q-1}$ by
\[ 
\phi^g_\ell(x,\gamma) = \left(p^{-\frac{1}{2}} (Z_{\frac{1}{b}}g)(x-\ell
  \frac{p}{q},\gamma+\frac{k}{p})\right)_{k=0}^{p-1} \ a.e. \ x,\gamma \in \mathbb{R}.
\] 
The following characterization of rationally oversampled Gabor frames
is due to Zibulski and Zeevi~\cite{MR1448221}.
\begin{theorem}
\label{thm:ZZ_singular_values}
  Let $A,B>0$, and let $g \in L^2(\R)$. Suppose $\mathcal{G}(g,a,b)$
  is a rationally oversampled Gabor system. Then the following
  assertions are equivalent:
\begin{enumerate}[(i)]
\item   $\mathcal{G}(g,a,b)$ is a Gabor frame for $L^2(\R)$ with bounds $A$ and $B$,
\item $\set{\phi^g_\ell(x,\gamma)}_{\ell=0}^q$ is a frame for $\C^p$ with
  uniform bounds $A$ and $B$ for a.e. $(x,\gamma) \in
  \itvcos{0}{1}^2$. 
\end{enumerate}
\end{theorem}

If $p=1$, i.e., $ab=1/q$, the Gabor system  $\mathcal{G}(g,a,b)$ is
said to be integer oversampled. By
Theorem~\ref{thm:ZZ_singular_values} it is a frame with bounds $A$ and
$B$ if and only if
\begin{align}
\label{eq:int-oversampl}
  A \le \left(\sum_{\ell =0}^{q-1}
    \absbig{Z_{\tfrac{1}{b}}g(x-\ell/q,\gamma)}^2\right)^{1/2} \le B
  \quad \text{for a.e. $x,\gamma \in \itvco{0}{1}^2$.}
\end{align}

If $g\in W(\R)$ is odd and continuous, then, by
Lemma~\ref{lem:Zak-symm}(ii),
$Z_{1/b}g(0,0)=Z_{1/b}g(\tfrac12,0)=0$ for any $b>0$, which by
\eqref{eq:int-oversampl} immediately implies
that $\gaborG{g}$ is not a
frame along the hyperbola $ab=\tfrac12$. Lyubarskii and
Nes~\cite{MR3027914} showed that this assertion extends to any of the
hyperbolas $ab= \tfrac{p}{p+1}$ for $p \in \N$ for any such odd window
function. The results in the remainder of this section show that
the frame property also must fail for certain $(a,b)$-values for
window functions with other symmetries formulated in terms of the
Fourier transform. We denote the new ``failure'' points in $\setprop{(a,b)\in
  \R^2_+}{ab<1}$ by $(a_i,b_i)$, $i=0,1,2,3,4$, where
\begin{equation}
a_i=b_i=\frac{1}{\sqrt{i+2}} \; (i=0,1,2), \quad  a_3=\frac{2}{\sqrt{3}}, b_3=\frac{1}{\sqrt{3}}  \quad a_4=\frac{1}{\sqrt{3}}, b_4=\frac{2}{\sqrt{3}}. \label{eq:non-frame-points}
\end{equation}
\begin{theorem}
\label{thm:4mp2}
Let $g \in H_2\cap W(\R)$. For any point $(a_i,b_i)$, $i \in
\set{0,1,3,4}$, as defined in \eqref{eq:non-frame-points}, the Gabor
system $\gaborG[a_i][b_i]{g}$ is not
 a frame for $L^2(\R)$, in particular, $\gaborG[a_i][b_i]{h_n}$ is not
 a frame for  $n=4m+2$, $m \in \N \cup
 \{0\}$. 
\end{theorem}

\begin{proof}
We consider first the assertion for $i=0$. Note that $a_0 b_0=1/2$, hence $\gaborG[a_0][b_0]{g}$ is an integer
  oversampled Gabor system with $p=1$ and $q=2$. By
  \eqref{eq:h-even-sqrt-2} in Lemma~\ref{lem:extension-identities}(i),
it follows that 
  $Z_{1/b_0}g(x-\ell/q,\gamma)=0$ for $\ell=0,1$ for
  $(x,\gamma)=(\tfrac{3}{4},\tfrac{1}{2})$. Since the Zak transform is
  continuous for  $g \in H_2\cap W(\R)$, we see that the lower bound
  in \eqref{eq:int-oversampl} cannot hold. Thus,
  $\gaborG[a_0][b_0]{g}$ is not a frame.

For the case $i=1$, we have $a_1 b_1=1/3$, hence $p=1$ and $q=3$. By
  \eqref{eq:h-even-sqrt-3} in Lemma~\ref{lem:extension-identities}(i),
  it follows that 
  $Z_{1/b_1}g(x-\ell/q,\gamma)=0$ for $\ell=0,1,2$ for
  $(x,\gamma)=(\tfrac{5}{6},\tfrac{1}{2})$. As before, this violates
  the frame property of $\gaborG[a_1][b_1]{g}$. 

For the case $i=3$, we have $a_3b_3=1/3$, hence $p=2$ and $q=3$. From case
$i=1$, we see that the matrix
$\Phi^g=\set{\phi^g_\ell(x,\gamma)}_{\ell=0}^q$ has a row of zeros. It
follows from  Theorem~\ref{thm:ZZ_singular_values} that
$\gaborG[a_3][b_3]{g}$ is not a frame. 

The assertion for $i=4$ follows from case $i=3$ by symmetry using Lemma~\ref{lem:Zak-symm}.
\end{proof}

\begin{theorem}
\label{thm:4mp3}
Let $g \in H_3\cap W(\R)$. For any  point $(a_i,b_i)$, $i \in
\set{1,2}$, as defined in \eqref{eq:non-frame-points}, the Gabor
system $\gaborG[a_i][b_i]{g}$ is not
 a frame for $L^2(\R)$, in particular, $\gaborG[a_i][b_i]{h_n}$ is not
 a frame for  $n=4m+3$, $m \in \N \cup
 \{0\}$. 
\end{theorem}

\begin{proof}
  We consider first the assertion for $i=1$. In this case $a_1b_1=1/3$ and
  $\gaborG[a_1][b_1]{g}$ is an integer oversampled Gabor system with
  $p=1$ and $q=3$. By Lemma~\ref{lem:extension-identities}(ii), 
  it follows that
  $Z_{1/b_1}g(x-\ell/q,\gamma)=0$ for $\ell=0,1,2$ for
  $(x,\gamma)=(\tfrac{2}{3},0)$. As in the proof of
  Theorem~\ref{thm:4mp2}, this shows that $\gaborG[a_1][b_1]{g}$
  cannot be a frame. The proof for $i=2$ we note that $Z_{1/b_2}g(x-\ell/q,\gamma)=0$ for $\ell=0,1,2,3$ for
  $(x,\gamma)=(\tfrac{3}{4},0)$
\end{proof}

Note that it also follows from
Lemma~\ref{lem:extension-identities}(ii) that
$(a_0,b_0), (a_3,b_3)$ and $(a_4,b_4)$ fall outside $\frameset{(g)}$
for $g \in H_3\cap W(\R)$. However, these obstructions are already
known by the results in \cite{MR3027914} since functions in $H_3\cap
W(\R)$ are odd.

From Theorem~\ref{thm:4mp2} we have four obstruction points for the
window class $H_2\cap W(\R)$. Theorem~\ref{thm:4mp3} provides us with
two new obstruction points for the window class $H_3\cap W(\R)$, not
already covered by the hyperbolic obstructions $ab=p/(p+1)$, $p \in
\N$. On the other hand, in general, no obstruction points can exist
for the class $H_0\cap W(\R)$ since it contains the Gaussian $h_0$. If
the conjecture by Lyubarskii and Nes~\cite{MR3027914} holds true, then
there are no general obstructions for the class $H_1\cap W(\R)$ in addition
to $ab=p/(p+1)$, $p \in \N$.

One might ask how badly the Gabor system fails to be a frame in the obstruction points. From the proofs above, it is clear that it is the lower frame bound that fails. In fact, any window in $W(\R)$ satisfy the upper frame bound.  The lower frame bound is a strong condition that is equivalent to injectivity and closedness of the range of the analysis operator $C_{g,a,b}: L^2(\R) \to \ell^2(\Z^2)$ defined by $C_{g,a,b} f=\set{\innerprod{f}{\myexp{2\pi i bm \cdot}g(\cdot-ak)}}_{k,m\in \Z}$. Note that injectivity of $C_{g,a,b}$ is equivalent with the Gabor system $\gaborG{g}$ being complete in $L^2(\R)$. For Hermite windows, Gr\"ochenig, Haimi, and Romero~\cite{GrochenigCompleteness2015} recently showed that, at least, completeness is guaranteed. To be precise, they proved as part of a more general result that, for any $n \in \N$, the system $\gaborG{h_n}$ is complete in $L^2(\R)$ for any rational $ab\le 1$. Hence, for each $(a_i,b_i)$, $i\in \set{0,1,2,3,4}$, given in \eqref{eq:non-frame-points}, the
Gabor system $\gaborG[a_i][b_i]{h_n}$, for $n=4m+2$ or $n=4m+3$, is a
complete Bessel system for which the lower frame bound is not
satisfied because the range of $C_{h_n,a_i,b_i}$ fails to be closed.


Even though both $(1/\sqrt{2},1/\sqrt{2}) \notin \frameset{(g)}$ and
$(1/\sqrt{3},1/\sqrt{3}) \notin \frameset{(g)}$ for $g \in H_2\cap
W(\R)$, no other points of the form $(1/\sqrt{k},1/\sqrt{k})$ can be
obstruction points for the frame property for the window class
$H_2\cap W(\R)$. In fact, by \cite{MR2292280,MR2529475} we know that
$ab<1/(n+1)$ is sufficient for the frame property of 
$\gaborG{h_n}$, hence, in particular, that
$\gaborG{h_2}$ is a frame for $ab<1/3$. Moreover, the obstruction
point $(a_1,b_1)=(1/\sqrt{3},1/\sqrt{3})$ shows that the region $ab<1/3$ is
sharp for $h_2$ in the sense that the smallest constant $c$ such that
$\setprop{(a,b)}{ab<c} \subset \frameset{(h_2)}$ is $c=1/3$. A similar
observation holds for $H_3\cap W(\R)$. In this case, the obstruction
point $(a_2,b_2)=(1/2,1/2)$ shows that the region $ab<1/4$ is sharp
for $h_3$. Since $ab<1$ and $ab<1/2$ are sharp for $h_0$ and $h_1$,
respectively, it is natural to ask if $ab<\tfrac{1}{n+1}$ is sharp for
$h_n$ for all $n\in \N$.


We have here focused on finding
$(a,b)$-values that serve as obstructions of the frame property
simultaneously for an
entire class of window functions. For a specific choice of a Hermite
function $h_n$, $n \ge 2$, one can most likely find many more new
obstructions; this is indeed indicated by the
numerical experiments in the next section.


\section{Numerical experiments}
\label{sec:numer-exper}

The numerical experiments in Matlab below use double precision
floating-point numbers. We
truncate the Hermite functions to obtain compactly supported functions
whenever the function value drops sufficiently low. 
This way a close
approximation to the Zak transform
 $Z_{1/b}h_n$ can be computed as a finite sum. We then discretize the
 Zak transform domain 
on a
 uniform sampling grid, e.g., $51 \times 51$. 
As we only consider integer oversampled Gabor systems, close
approximations to the frame bounds
are easily computed for given values of $a$ and $b$ using the
formula~\eqref{eq:int-oversampl}. The approximated bounds
$A_{\mathrm{apx}}$ and $B_{\mathrm{apx}}$ will (up to machine
precision) be 
larger and smaller, respectively, than the true optimal frame bounds from \eqref{eq:int-oversampl}, i.e.,
$A_{\mathrm{opt}} \le A_{\mathrm{apx}} \le B_{\mathrm{apx}} \le B_{\mathrm{opt}}$. 

\begin{example}
\label{exa:h2}
  Let us first illustrate Theorem~\ref{thm:4mp2} for $h_2$.
  Figure~\ref{fig:plots} shows that the upper and lower frame bound of
  $\mathcal{G}(h_2,a,b)$ along $ab=1/2$ for $b \in
  \itvcc{\tfrac18}{4}$. 
  \begin{figure}[!h]
    \centering \input{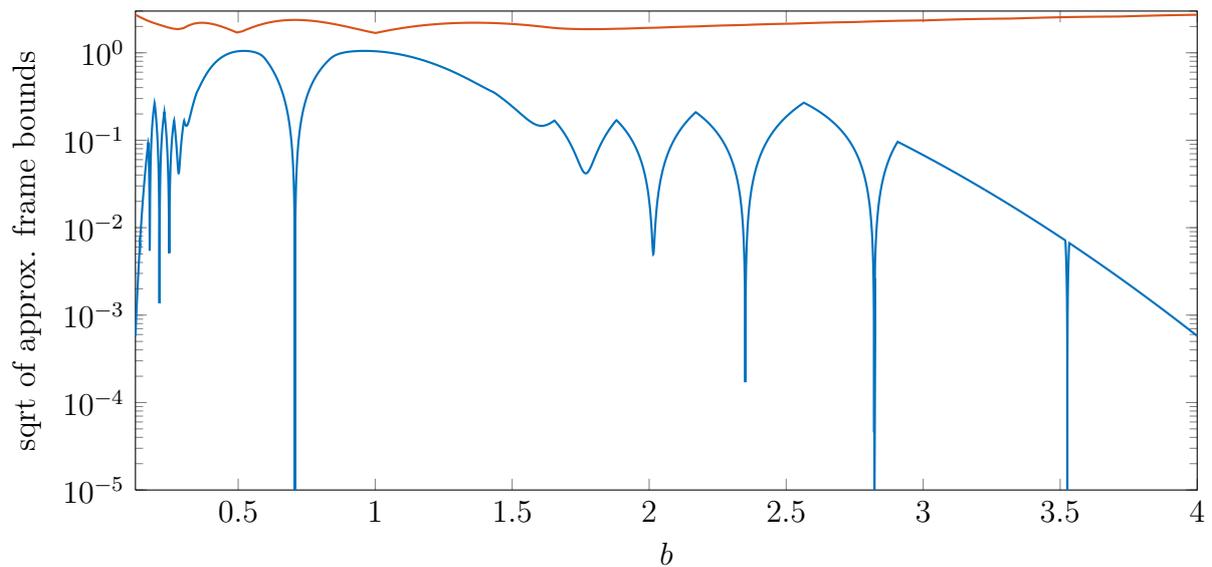}
\caption{Numerical approximations of the upper (red)
  and lower (blue) frame bound for
          $\mathcal{G}(h_2,a,b)$ along $ab=1/2$. At the point
          $(a_0,b_0)=(1/\sqrt{2},1/\sqrt{2})$ the estimate of $\sqrt{A}$
          essentially drops to machine precision $\approx 7 \cdot 10^{-16}$ (not
          shown).
        }
        \label{fig:plots}
      \end{figure}
We first remark that $A_{\mathrm{apx}}^{1/2}$
  drops to machine precision at
  $(a_0,b_0)=(1/\sqrt{2},1/\sqrt{2})$. Note also that the frame bounds
  are symmetric about $b=1/\sqrt{2}$ according to
  Lemma~\ref{lem:Zak-symm}, that is, $\mathcal{G}(h_2,1/(2b),b)$ and
  $\mathcal{G}(h_2,b,1/(2b))$ have the same frame bounds.

      The behavior of ${A_{\mathrm{apx}}^{1/2}}$ is rather
      complicated. The drops of ${A_{\mathrm{apx}}^{1/2}}$ below, say,
      $10^{-3}$, are very narrow and therefore difficult to resolve
      due to the discretization of the $b$ range. Moreover, it is unclear if
      $(a_0,b_0)=(1/\sqrt{2},1/\sqrt{2})$ is the only point along
      $ab=1/2$ that does not belong to $\frameset{(h_2)}$. Around
      $b=2.35$ and $b=2.82$ the values of $A_{\mathrm{apx}}^{1/2}$ are
      in the order of $10^{-4}$ and $10^{-7}$, respectively. At
      $b=3.5261848971734$ the value of $A_{\mathrm{apx}}^{1/2}$ even
      drops to $1.6 \cdot 10^{-12}$, however, it does not drop below
      this value even when the discretization is refined. There may
      very well exist a $(a,b)$-pair near the point
      $(1/(2b),b)$, where $b=3.5261848971734$, for which $\gaborG{h_2}$
      is not a frame. In any event, since $A_{\mathrm{apx}}^{1/2}
      \approx 10^{-12}$, i.e., $A_{\mathrm{apx}} \approx
      10^{-24}$, such a Gabor system is 
      badly conditioned and should not be used for numerical purposes.
    \end{example}


  Let us end this paper with two examples not covered by the results
  in Section~\ref{sec:obstruc}. 

\begin{example}
\label{exa:h4}
In this example we consider Gabor systems generated by $h_4$ and
$h_5$. Note that these functions belong to $H_0$ and $H_1$,
respectively. Recall that no obstructions of the frame property is
known of $h_4$, while the hyperbolas $ab=\tfrac{p}{p+1}$, $p \in \N$,
are the only known obstructions for $h_5$. Figure~\ref{fig:plot-H4-H5}
shows the approximated frame bounds of $\mathcal{G}(h_4,a,b)$ along
$ab=1/2$ and of $\mathcal{G}(h_5,a,b)$ along $ab=1/3$. The general
behavior is similar to that of $h_2$ in Figure~\ref{fig:plots}. For
$\mathcal{G}(h_4,a,b)$ the lower frame bound $A_{\mathrm{apx}}^{1/2}$
drops to machine precision four times in the considered $b$
range. This behavior can be explained as follows. In Maple one can verify with
arbitrary precision that 
\begin{equation}
  Z_{\sqrt[4]{3}} h_4(0,\tfrac{1}{2}) \defeq 3^{1/8} \sum_{k \in \Z}
  (-1)^k h_4(3^{1/4}k) = 0
\label{eq:id-h4}
\end{equation}
holds. Recall that the Zak transform also has a zero at
$(\tfrac12,\tfrac12)$ since $h_4$ is even. Hence,
equation~\eqref{eq:id-h4} implies that the lower bound in
\eqref{eq:int-oversampl} is violated for
$(x,\gamma)=(\tfrac12,\tfrac12)$. Therefore, $\gaborG{h_4}$ is not
frame for $(a,b)=(3^{1/4}/2,3^{-1/4})$, and by
symmetry using Lemma~\ref{lem:Zak-symm}, also not for
$(a,b)=(3^{-1/4},3^{1/4}/2)$. 
Similarly,  one can verify in Maple with
arbitrary precision that 
\begin{equation}
  Z_{\tfrac{1}{\sqrt[4]{3}}} h_4(0,\tfrac{1}{2}) \defeq 3^{-1/8} \sum_{k \in \Z}
  (-1)^k h_4(3^{-1/4}k) = 0
\label{eq:id-h4-2}
\end{equation}
holds. Equation~\eqref{eq:id-h4-2} implies that $\gaborG{h_4}$ is not
frame for $(a,b)=(3^{-1/4}/2,3^{1/4})$, and by
symmetry, also not for
$(a,b)=(3^{1/4},3^{-1/4}/2)$. A proof of the
identities~\eqref{eq:id-h4} and \eqref{eq:id-h4-2} must rely on other
methods that used in Section~\ref{sec:some-infinite-series} since the
Gaussian $h_0$ does not satisfy the identities and since both $h_0$ and
$h_4$ belong to $H_0$.   
 \begin{figure}[!h]
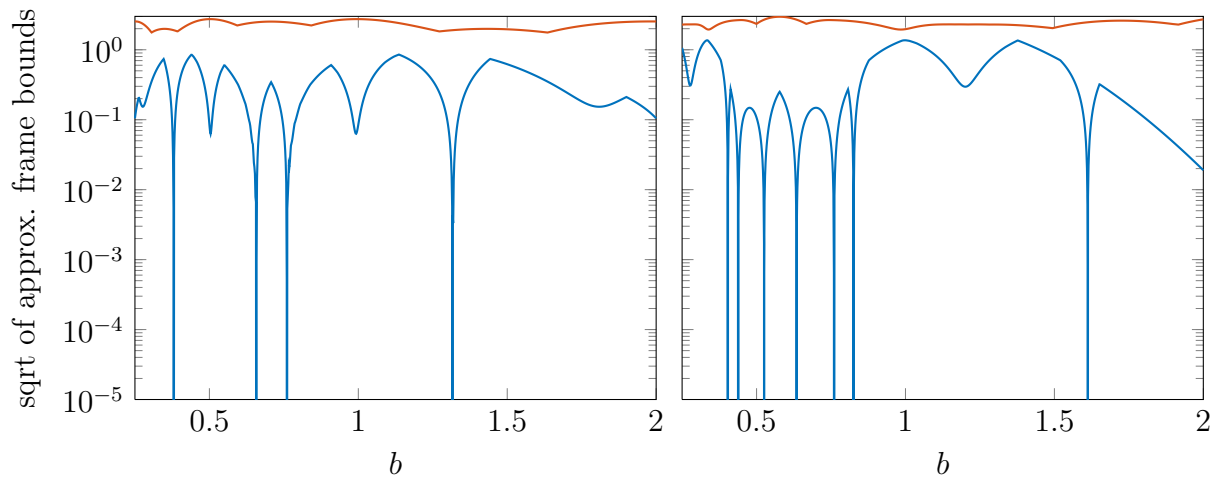

\definecolor{mycolor1}{rgb}{0.00000,0.44700,0.74100}%
\definecolor{mycolor2}{rgb}{0.85000,0.32500,0.09800}%

\begin{minipage}[t]{.5\textwidth}
\centering
 \input{figH4}
    \end{minipage}
\hspace{.2em} 
\phantom{e}
    \begin{minipage}[t]{0.45\textwidth}
        \centering
\input{figH5}
\end{minipage}
\vspace*{-.5em}
\caption{Numerical approximations of the square root of
  the upper (red)
  and lower (blue) frame bound for
          $\mathcal{G}(h_4,a,b)$ along $ab=1/2$ (left) and
          $\mathcal{G}(h_5,a,b)$ along $ab=1/3$ (right). In all
          instances, 
          where $A_{\mathrm{apx}}^{1/2}$ drops below $10^{-5}$, it drops
        to a value in the order of machine precision $\approx 10^{-16}$
        (not shown). }
\label{fig:plot-H4-H5}
 \end{figure}

For the $5$th Hermite function $h_5$ the lower frame bound
$A_{\mathrm{apx}}^{1/2}$ drops to machine precision seven times along
$ab=1/3$ in the considered range in
Figure~\ref{fig:plot-H4-H5}. Here, similar arguments as for $h_4$ can be used
to show that
$(a,b)=(\sqrt[4]{27}/3,1/\sqrt[4]{27})$ and $(a,b)=(1/\sqrt[4]{27},\sqrt[4]{27}/3)$ 
do not belong to  $\frameset{(h_5)}$. 
Indeed, one can verify in Maple with
arbitrary precision that
\[
 Z_{\sqrt[4]{27}}h_5(\tfrac{p}{3},\tfrac12) = 0 \quad \text{for } p\in \set{0,1,2}.
\]
Similar identities that explain the other
five drops of $A_{\mathrm{apx}}^{1/2}$ for
$\mathcal{G}(h_5,a,b)$ most likely exist.
\end{example}

In Example~\ref{exa:h4} above we briefly considered obstructions of
the frame property for Hermite functions
outside the two classes $H_2$ and $H_3$. The
methods developed in this paper for Wiener space functions in the
eigenspaces $H_2$ and $H_3$
rely only on the corresponding eigenvalue of the Fourier transform.
However, since
both $h_0 \in H_0$ and $h_4 \in H_0$ have the same eigenvalue, namely
$1$, it is obvious that other methods are needed if one attempts to
disprove Gr\"ochenig's conjecture for, say, all functions in $\setprop{h_{4m}}{m
  \in \N}$.
  


\def\cprime{$'$} \def\cprime{$'$} \def\cprime{$'$} \def\cprime{$'$}
  \def\uarc#1{\ifmmode{\lineskiplimit=0pt\oalign{$#1$\crcr
  \hidewidth\setbox0=\hbox{\lower1ex\hbox{{\rm\char"15}}}\dp0=0pt
  \box0\hidewidth}}\else{\lineskiplimit=0pt\oalign{#1\crcr
  \hidewidth\setbox0=\hbox{\lower1ex\hbox{{\rm\char"15}}}\dp0=0pt
  \box0\hidewidth}}\relax\fi} \def\cprime{$'$} \def\cprime{$'$}
  \def\cprime{$'$} \def\cprime{$'$} \def\cprime{$'$} \def\cprime{$'$}


\end{document}